\documentclass[11pt, oneside]{article} 
\usepackage{geometry} 
\usepackage{graphicx}	
\usepackage{amsmath}
\usepackage{amsthm}
\usepackage{amssymb}
\usepackage{color}
\usepackage{todonotes}
\usepackage[mathscr]{euscript}

\usepackage{comment}

\newcommand{\norm}[1]{\left\lVert#1\right\rVert}
\newcommand{\innerpr}[1]{\langle\,#1\,\rangle}
\newcommand{\Int}{\text{Int}}
\newcommand{\Id}{\text{Id}}

\newcommand{\mat}[1]{\begin{pmatrix} #1 \end{pmatrix}}
\newcommand{\D}{\mathcal{D}}
\newcommand{\M}{\mathscr{M}}
\newcommand{\R}{\mathbb{R}}
\newcommand{\Z}{\mathbb{Z}}
\newcommand{\ds}{\displaystyle}

\newtheorem{theorem}{Theorem}

\newtheorem{corollary}[theorem]{Corollary}

\newtheorem*{maintheorem}{Main Theorem}
\newtheorem*{definition*}{Definition}

\title{A new proof of the existence of embedded surfaces with Anosov geodesic flow}
\author{Victor Donnay\thanks{Department of Mathematics, Bryn Mawr College, Bryn Mawr, PA, USA,  vdonnay@brynmawr.edu}, Daniel Visscher\thanks{Department of Mathematics, Ithaca College, Ithaca, NY, USA, dvisscher@ithaca.edu. \newline \indent The second author was supported by a Summer Research Grant from Ithaca College.}}
\date{}

\begin{document}
\maketitle

\begin{abstract}
We give a new proof of the existence of compact surfaces embedded in $\R^3$ with Anosov geodesic flows. This proof starts with a non-compact model surface whose geodesic flow is shown to be Anosov using a uniformly strictly invariant cone condition.  Using a sequence of explicit maps based on the standard torus embedding, we produce compact embedded surfaces that can be seen as small perturbations of the Anosov model system and hence are themselves Anosov. 
\end{abstract}

The geodesic flow on a surface of negative Gaussian curvature serves as a prototype of strongly chaotic deterministic dynamical systems. In the 1960's, Anosov abstracted the dynamical property of uniform hyperbolicity in such systems~\cite{Anosov}, starting a fruitful and still active area of research in smooth dynamical systems. Uniformly hyperbolic (``Anosov'') systems on compact manifolds are known to be stably ergodic, mixing, and to have exponential decay of correlations (see~\cite{Dolgopyat,Liverani} and references therein).

The dynamics of a geodesic flow is determined by the metric that generates it. Riemannian metrics can be determined intrinsically (as a smoothly varying inner product on the tangent space) or extrinsically (as the restriction of a metric in an ambient space that the surface is embedded in). A physically intuitive way to get a metric on a surface is to embed it in $\R^3$, defining the metric extrinsically as the $\R^3$ Euclidean metric restricted to the surface. We will call surfaces with such metrics \emph{isometrically embedded in $\R^3$}.

The geometric property that generates the uniform hyperbolicity of Anosov's examples is negative curvature. Such surfaces, however, if they are compact, cannot be isometrically embedded in $\R^3$, since they would necessarily contain points of positive curvature (for instance, the points of tangency with the smallest bounding sphere). This raises a natural question, originally asked by Michael Herman: do there exist compact surfaces isometrically embedded in $\R^3$ with Anosov geodesic flows? 

Herman's question was answered affirmatively in~\cite{Donnay-Pugh-embeddedsurfaces} by Donnay and Pugh, who gave a process for increasing the genus of a particular embedded surface in a way that adds more areas of negative curvature while decreasing the value of the positive curvature.
Eventually, this process must yield an Anosov geodesic flow on a surface of large (but unspecified) genus. The genus of a surface is a simple way to quantify its topological complexity; it also gives an obstruction to a surface generally supporting an Anosov geodesic flow. Klingenberg showed in~\cite{Klingenberg} that any metric with conjugate points cannot have an Anosov geodesic flow (see also the proof by Ma\~n\'e~\cite{Mane}). Thus, the sphere and the torus cannot support a metric with an Anosov geodesic flow, regardless of whether the metric comes from an isometric embedding. This is the extent of the topological obstruction in general though: all surfaces of genus $g\geq 2$ support metrics with Anosov geodesic flow, the simplest examples being metrics of constant negative curvature. In other words, all surfaces with at least a little topological complexity ($g\geq 2$) support an Anosov geodesic flow.  A refinement of Herman's question, then, is: how topologically complex must a surface be in order to support an isometrically embedded Anosov geodesic flow?\footnote{Kourganoff shows that any orientable surface of genus $\geq 11$ can be isometrically embedded in $\mathbb{S}^3$ so that it has an Anosov geodesic flow~\cite{Kourganoff}. His methods do not work for isometric embeddings into Euclidean $\R^3$.}

In this paper, we give an alternate proof of the following theorem with a view toward quantifying our techniques.

\begin{maintheorem}[\cite{Donnay-Pugh-embeddedsurfaces}]\label{mainthm:exist}
There exist compact  embedded surfaces in $\mathbb{R}^3$ for which the geodesic flows are Anosov. 
\end{maintheorem}

Our proof uses a global argument, rather than the local patches of~\cite{Donnay-Pugh-embeddedsurfaces}, with a simple explicit embedding map based on the standard torus embedding. In a sequel paper, we quantify these procedures to give a non-optimal lower bound on the genus for which a surface supports an isometrically embedded Anosov geodesic flow. Preliminary estimates put the genus at roughly $10^{9}$.

Heuristically, we relate the geometry and dynamics of the Main Theorem as follows. Negative curvature causes families of geodesics to diverge from each other, a property that generates hyperbolicity. Once diverging, a family of geodesics will remain diverging in non-positive curvature. Such a family will also remain diverging through a region of small positive curvature as long as the positive curvature does not ``outweigh'' the negative curvature. These properties are all demonstrated by simple Jacobi field computations. In this scheme, one can attempt to construct a surface with Anosov geodesic flow by arranging the geometry so that each geodesic encounters ``more'' negative curvature than positive curvature.

The total amount of curvature on a surface $S$ is related to its genus via the Gauss-Bonnet Theorem:
    $$
    \int_S\, K ~dA = 2 - 2 g.
    $$ 
This means that the larger the genus of the surface, the more negative curvature we have at our disposal to generate hyperbolicity. The conjugate point obstruction, however, shows that a lot of negative curvature is not enough to produce an Anosov geodesic flow---one must also consider the effects of any positive curvature on the surface. An isometrically embedded surface necessarily has regions of positive curvature. If there is too much positive curvature in any one region, that region will produce a conjugate point and thereby the geodesic flow will not be Anosov regardless of how much negative curvature there is elsewhere. Thus, the positive curvature must be carefully interspersed with regions of negative curvature, with the negative curvature preventing the development of conjugate points. This motivates a geometric version of the refinement of Herman's question above: how can negative and positive curvature be arranged on an embedded surface so that the resulting geodesic flow is Anosov? And how efficiently can one do this---i.e., what is the smallest amount of negative curvature (or the smallest genus) that can be used?

In~\cite{Donnay-Pugh-embeddedsurfaces}, Donnay and Pugh give a proof of this theorem by triangulating a round sphere and placing a large number of dispersing tubes along the boundary of each triangle, modifying their construction in~\cite{Donnay-Pugh-finitehorizon}. In this paper, we attach a small number of dispersing tubes to a flat torus, and then repeat this pattern of tubes in a periodic fashion to create a non-compact model space 
with a model metric that comes from the Euclidean $\R^3$ metric. The geodesic flow on this surface is easily seen to be Anosov (Theorem~\ref{thm:modelanosov}). It is well known that the set of metrics with Anosov geodesic flows is $C^2$-open for compact surfaces. For our non-compact model surface, we use a uniformly strictly invariant cone condition to show that there is such an open set around our model metric (Corollary~\ref{cor:perturbations}). We then consider a sequence of covering maps from the model surface to compact embedded surfaces in $\R^3$ (Theorem~\ref{thm:periodic}). This sequence yields metrics 
that converge to the model space metric (Theorem~\ref{thm:converge}), and thus at some point the sequence must enter the open set of Anosov metrics around the model metric.

One can, in principle, estimate the genus of the surface constructed here by quantifying the components of the proof: determining how large the open set of Anosov metrics around the model metric is (i.e., how strongly Anosov the model geodesic flow is) and how close elements of the sequence of converging metrics are to the model metric. 
We undertake this analysis in a sequel to this paper.

\section{Anosov geodesic flows}
Let $M$ be a surface with Riemannian metric $g$, and let $SM$ denote the sphere bundle\footnote{I.e., $SM = TM/\sim$, where $(p,v) \sim (q,w)$ iff $p=q$ and $v = c w$ for some $c \not= 0$. This is naturally identified with the unit tangent bundle, but since the notion of ``unit'' depends on a metric and we will consider different metrics on the same manifold, we will use $SM$ to mean the sphere bundle.}. We will write elements of $SM$ as $x = (p,v)$, with $p \in M$ and $v$ a unit vector in $T_p M$. The \emph{geodesic flow} of $(M,g)$ is the flow $\varphi^t: SM \to SM$ defined by 
    \[
    \varphi^t(x) = \varphi^t(p,v) = (\gamma_{(p,v)}(t), \dot{\gamma}_{(p,v)}(t)),
    \]
where $\gamma_{(p,v)}$ is the unique geodesic on $(M,g)$ with $\gamma_{(p,v)}(0)=p$ and $\dot{\gamma}_{(p,v)}(0)=v$.
The geodesic flow $\varphi^t$ on $SM$ is \emph{Anosov} if there is a flow-invariant splitting $T(SM) = \langle \dot{\varphi}^t \rangle \oplus E^s \oplus E^u$ and constants $C,\lambda > 1$ such that $\langle \dot{\varphi}^t \rangle$ is the direction of the flow and
    \begin{equation}\label{eqn:Anosovdefn}
    \norm{D_x \varphi^t |_{E^s(x)}} < C\lambda^{-t}
~~~~\text{and}~~~~
    \norm{D_x \varphi^t |_{E^u(x)}} > C^{-1} \lambda^t
    \end{equation}
for $t \geq 0$ and all $x \in SM$, with the norm determined by the Sasaki metric.  Note that, since $\varphi^t$ is invertible, these conditions are equivalent to $\norm{D_x \varphi^{-t} |_{E^s(x)}} > C^{-1} \lambda^t$ and $\norm{D_x \varphi^{-t} |_{E^u(x)}} < C\lambda^{-t}$, respectively.

By the definition above, $E^u$ and $E^s$ must be non-trivial and each of the subspaces $\langle \dot{\varphi}^t \rangle (x)$, $E^s(x)$, and $E^u(x)$ over the point $x \in SM$ must be one-dimensional. The geodesic flow preserves length along $\dot{\varphi}^t$, so the expansion and contraction conditions on $E^s(x)$ and $E^u(x)$ imply that they must be perpendicular to $\langle \dot{\varphi}^t \rangle (x)$. Since the geodesic flow is volume preserving\footnote{The Sasaki metric on $SM$ gives a volume measure which is the product of the area measure on $M$ induced by the metric $g$ and Lebesgue measure on the $\mathbb{S}^1$ fibers.}, it must also preserve area on the perpendicular subspace $\langle \dot{\varphi}^t \rangle^\perp$. 
In summary, an Anosov splitting decomposes each $\langle \dot{\varphi}^t \rangle^\perp$ above a point $x \in SM$ into two one-dimensional subspaces in such a way that the decomposition is invariant under $D\varphi^t$ and the subspaces have the expansion/contraction properties above.

There is another natural geometric splitting of the perpendicular subspace $\langle \dot{\varphi}^t \rangle^\perp = H \oplus V$ that comes along with the metric $g$. The subspace $V$ is called the \emph{vertical subspace} and consists of the vectors at $x=(p,v)$ that are derived from varying the vector $v$ (while keeping $p$ fixed). The subspace $H$ is called the \emph{horizontal subspace} and consists of the vectors at $x=(p,v)$ that are derived from varying the basepoint $p$ on the manifold perpendicular to the flow direction (while keeping $v$ fixed). Note that, in order to keep the vector $v$ ``fixed'' while varying $p$, one needs a notion of parallel transport from a metric, so the subspace $H$ depends on the metric used. Both $H$ and $V$ are one-dimensional over the point $x$, and $T(SM) = \langle \dot{\varphi}^t \rangle \oplus H \oplus V$. This, however, cannot be an Anosov splitting: while the subspace $H$ may be invariant under $D\varphi^t$ (if the metric is flat), the vertical subspace $V$ is never invariant under a geodesic flow. 

The horizontal and vertical subspaces give rise to a set of coordinates, which, via Jacobi fields,  provide a natural lens through which to view the dynamics of the geodesic flow. In particular, this aids in showing the invariance of a cone field under the geodesic flow. 
For each $x \in SM$, let $\xi_h (x) \in H$ and $\xi_v (x) \in V$ be unit vectors such that $(\dot{\varphi}^t, \xi_h, \xi_v)$ is in the standard orientation. Using these vectors as a basis for $T_x(SM)$, any $\omega \in T_x(SM)$ can be written as $\omega = c_1 \dot{\varphi}^t + c_2 \xi_h + c_3 \xi_v$; its norm with respect to the Sasaki metric then satisfies
$\norm{\omega}^2 = c_1^2 + c_2^2 + c_3^2$. 

A vector $\xi$ in the perpendicular subspace of $T_x(SM)$ evolves according to the Jacobi equation in the following way. Let $\xi = j(0)\xi_h + j'(0)\xi_v$ be identified with the initial conditions of a solution to the Jacobi equation. Then
$$D_x\varphi^t \xi = j(t) \xi_h + j'(t) \xi_v$$
where $j(t)$ satisfies the Jacobi equation $j(t)''+K(\gamma_x(t))j(t)=0$.

The cone criterion of Theorem~\ref{thm:cone} below provides a way to produce an Anosov splitting.
Given vectors $X$ and $Y$ in $\langle \dot{\varphi}^t \rangle^\perp$ based at $x$ with $\norm{X}$ and $\norm{Y}> 0$, define a \emph{cone} to be the set
    \[
    C_{X,Y}(x) = \{ \alpha X + \beta Y ~|~ \alpha\beta \geq 0 \}.
    \]
Define the angle $\theta$ of the cone $C_{X,Y}(x)$ to be the angle between two vectors that determine it: 
    \[
    \theta(C_{X,Y}(x)) = \arccos \left( \frac{\innerpr{X,Y}}{\norm{X}\norm{Y}} \right).
    \]
Observe that scaling $X$ or $Y$ by a positive constant does not change either the cone or the angle, so we will henceforth assume that $X$ and $Y$ are unit vectors. 
A cone field on $SM$ is a collection of cones $C_{X,Y}(x)$ defined for every point $x \in SM$.
\begin{definition*}
A cone field $C_{X,Y}$ on $SM$ is
\begin{enumerate}
\item[a.] \emph{$\tau$-invariant} if
    $
    D\varphi^\tau (C_{X,Y}(x)) \subseteq C_{X,Y}(\varphi^\tau x)
    $
for all $x \in SM$, and it is 
\item[b.] \emph{$\tau$-strictly invariant}  if   
    $
    D\varphi^\tau (C_{X,Y}(x)) \setminus \{(0,0)\} \subset \Int\left(C_{X,Y}(\varphi^\tau x)\right)
    $
for all $x \in SM$.
\end{enumerate}

\noindent For a $\tau$-strictly invariant cone field, define
$\Delta \theta: SM \to \R$  by 
    \[
    \Delta \theta (x) = \frac{\theta(D\varphi^\tau (C_{X,Y}( \varphi^{-\tau}x)))}{\theta(C_{X,Y}(x))}.
    \]
Note that the assumption $D\varphi^\tau (C_{X,Y}(\varphi^{-\tau}x)) \subset \Int \left(C_{X,Y}(x)\right)$ implies that $\Delta \theta < 1$. 
A $\tau$-strictly invariant cone field is
\begin{enumerate}
\item[c.] \emph{uniformly $\tau$-strictly invariant} if there exists a $c<1$ such that $\Delta \theta(x) \leq c$ for all $x \in SM$.
\end{enumerate}
\end{definition*}
\noindent A \emph{continuous cone field} $C_{X,Y}$ is a continuously varying set of cones above each point $x \in SM$ (e.g., coming from continuous vector fields $X$ and $Y$ on $SM$). 
Note that, for a fixed continuous cone field, we can regard $\theta$ as a continuous function on $SM$.

The following cone criterion is similar to those used in other dynamical settings (see, e.g.,~\cite{Katok-Hasselblatt,Kourganoff-unifhyp}). 

\begin{theorem}\label{thm:cone}
Let $M$ be a surface with bounded curvature and let $C_{X,Y}$ be a uniformly $\tau$-strictly invariant cone field on $SM$.
Then the geodesic flow $\varphi^t$ is Anosov.
\end{theorem}

\begin{proof}
Let 
    \[
    E^u(x) = \bigcap_{n \geq0} D\varphi^{n\tau} \left( C_{X,Y}( \varphi^{-n\tau}x) \right),
    \]
which is non-empty since $C_{X,Y}$ is $\tau$-invariant. Set 
    \[
    \Delta \theta^n (x) = \frac{\theta(D\varphi^{n\tau} (C_{X,Y}( \varphi^{-n\tau}x)))}{\theta(C_{X,Y}(x))}.
    \]
Then 
    \[
    \Delta \theta^n (x) = \Delta \theta (\varphi^{(-n+1)\tau}x) \cdots \Delta \theta (\varphi^{-\tau}x) \cdot \Delta \theta (x)  \leq c^n,
    \]
so that $\lim_{n\to \infty} \Delta \theta^n = 0$. Hence, by the Nested Interval Theorem, $E^u(x)$ is a one-dimensional subspace of $\langle \dot{\varphi}^t \rangle^\perp (x)$.

For the stable direction $E^s$, we apply the backwards flow to the complement cone field
$C^c_{X,Y} = \{ \alpha X + \beta Y ~|~ \alpha\beta \leq 0 \}$. The uniform $\tau$-strict invariance property of $C_{X,Y}$ implies the same of $C^c_{X,Y}$.
Then the arguments above show that
    \[
    E^s(x) = \bigcap_{n\geq 0} D\varphi^{-n\tau} \left( C^c_{X,Y}(\varphi^{n\tau}x) \right)
    \]
is a one-dimensional subspace of $\langle \dot{\varphi}^t \rangle^\perp$.

Now we show that the volume preserving property of $D\varphi^t$ along with the uniform contraction of the width of the cones force uniform expansion in the $E^u$ direction. The subspace $E^u(x)$ is spanned by the vector $\xi^u = X+\alpha Y$ for some $\alpha$, with $\alpha$ uniformly bounded away from 0 for all $x \in SM$. Let $X^n = D\varphi^{n\tau} X$, let $Y^n = D\varphi^{n\tau} Y$, and let $\theta^n (x) = \theta(D\varphi^{n\tau}(C_{X,Y}(x))) = \theta(C_{X^n,Y^n}(\varphi^{n\tau}x))$. Note that $\theta^n \to 0$ uniformly for $x \in SM$, since the cone field is uniformly $\tau$-strictly invariant and 
    \[
    \theta^n(x)= \Delta \theta^n (\varphi^{n\tau} x)\cdot \theta  (C_{X,Y}(\varphi^{n \tau} x)) < c^n \theta  (C_{X,Y}(\varphi^{n \tau} x)) < c^n \pi
    \]
for the constant $c<1$ from above.

The area preserving property of $D\varphi^t$ on $\langle \dot{\varphi}^t \rangle^\perp$ means that 
\begin{equation*}\label{eqn:areapres}
\norm{X}\norm{\alpha Y}\sin \theta = \norm{X^n}\norm{\alpha Y^n}\sin \theta^n.
\end{equation*}
Fix a $k >1$. Since $\sin \theta^n \to 0$, either $\norm{X^n}$ or $\norm{\alpha Y^n}$ is eventually greater than $k\norm{X+\alpha Y}$. Using the fact that $\theta^n$, the angle between $X^n$ and $\alpha Y^n$, goes to $0$, we get that for large enough values of $n$,
    \begin{equation}
    \norm{D\varphi^{n\tau} \xi^u} = \norm{X^n + \alpha Y^n} > \max \left\{ \norm{X^n},\norm{\alpha Y^n} \right\} > k\norm{X+\alpha Y} = k\norm{\xi^u}. \label{eqn:expansion}
    \end{equation}
Uniform $\tau$-strict invariance implies that all of the above estimates are uniform for all $x \in SM$. Hence, (\ref{eqn:expansion}) implies uniform expansion for the time $\tau$ map. The boundedness of the curvature on $M$ implies that $\norm{D_x\varphi^t}$ is bounded uniformly for all $t$ between $0$ and $\tau$ and for all $x \in SM$. Thus, uniform expansion for the time $\tau$ map implies the uniform expansion condition~(\ref{eqn:Anosovdefn}) for the flow. A similar argument holds for the uniform contraction of vectors in the stable subspace. 

The above argument produces a one-dimensional $E^u(x)$ satisfying the uniform expansion condition~(\ref{eqn:Anosovdefn}) for every $x \in SM$. The bundle $E^u$ is flow invariant since, due to the area preserving property of the flow on the perpendicular subspace, there can only be one direction in $\langle \dot{\varphi}^t \rangle^\perp$ that contracts uniformly in backwards time. The stable bundle $E^s$ is similarly flow invariant.
\end{proof}

\section{Model space}

Let $\D \subset \R^2$ be a collection of closed disks with $\Z^2$ translational symmetry with the property that every geodesic that starts outside a disk will eventually enter a disk (for example, see Figure~\ref{fig:disks}). Then $\D$ satisifes the following   finite horizon property. 

\begin{figure}[ht]
\begin{center}
\includegraphics[width=2.5in]{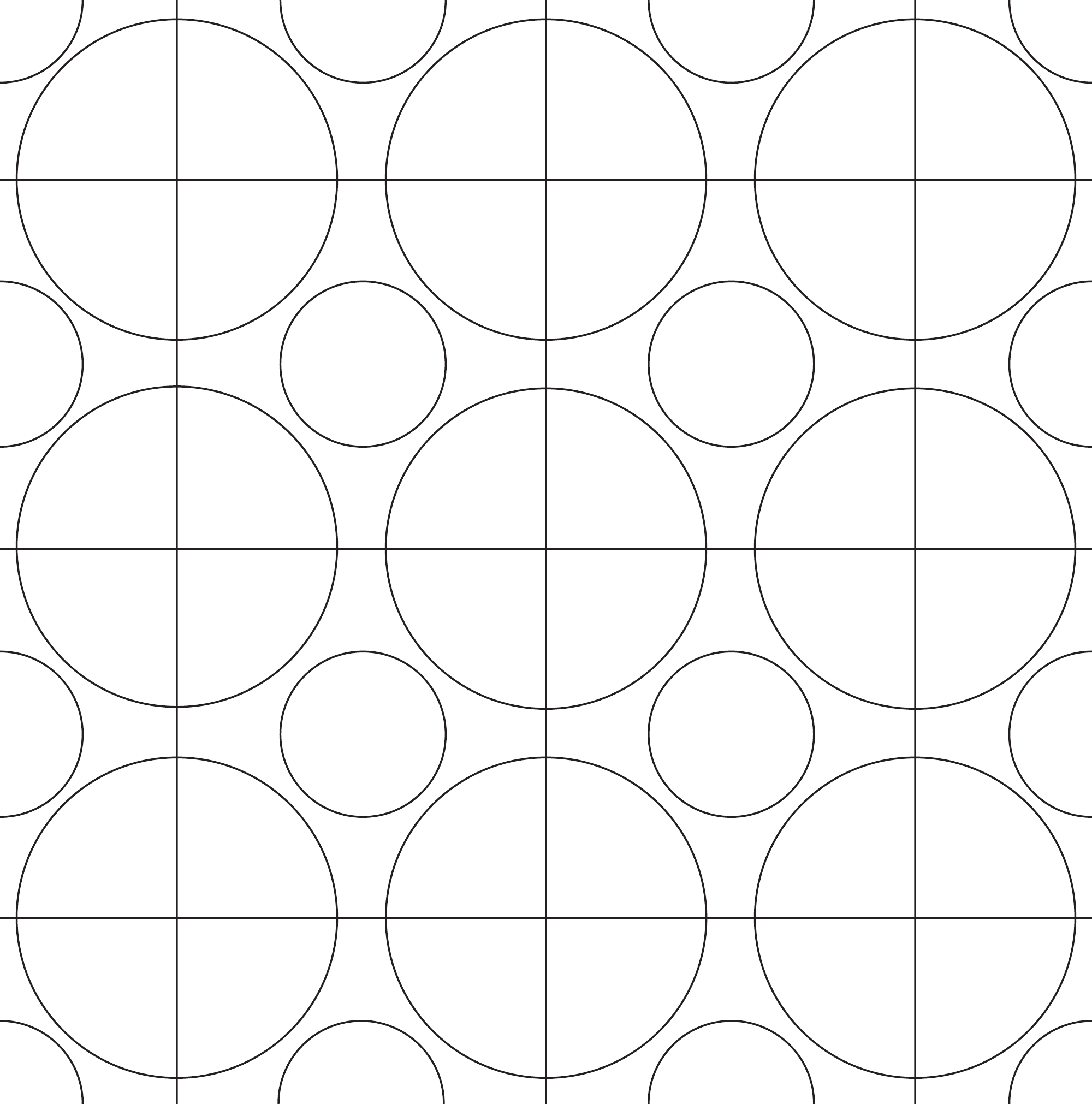}
\end{center}
\caption{A set $\D$ of disks on the plane with $\Z^2$ translational symmetry and the finite horizon property.}\label{fig:disks}
\end{figure}

\begin{definition*}
A set $S \subset (M,g)$ has the \emph{finite horizon property} if there exists $T>0$ such that for all $x \in SM$, there is a time $t$ with $0<t<T$ with $\gamma_{x}(t) \subset \Int \, (S)$.
\end{definition*}

We construct a surface $\M \subset \R^3$ based on $\D$. Denote by $(u, v, w)$ the standard Euclidean coordinates on $\R^3$, and tile the planes $w=0$ and $w=1$ with copies of $\D$. 
Smoothly attach negatively curved surfaces of revolution of height 1 along these disks (one may need a few different such tubes for different sized disks in $\D$) in the $w=0$ and $w=1$ planes, as in Figure~\ref{fig:modelspace}. The resulting surface $\M$ is $C^\infty$-smooth and non-compact with infinite genus. Since it is embedded in $\R^3$, $\M$ has an extrinsically defined metric $g^\M_0$ that is the restriction of the Euclidean metric $g_0$ on $\R^3$. In this metric, the curvature is strictly negative on the interior of the surfaces of revolution and zero outside of the disks, hence non-positive everywhere.
\begin{figure}[ht] 
\begin{center}
\includegraphics[width=5in]{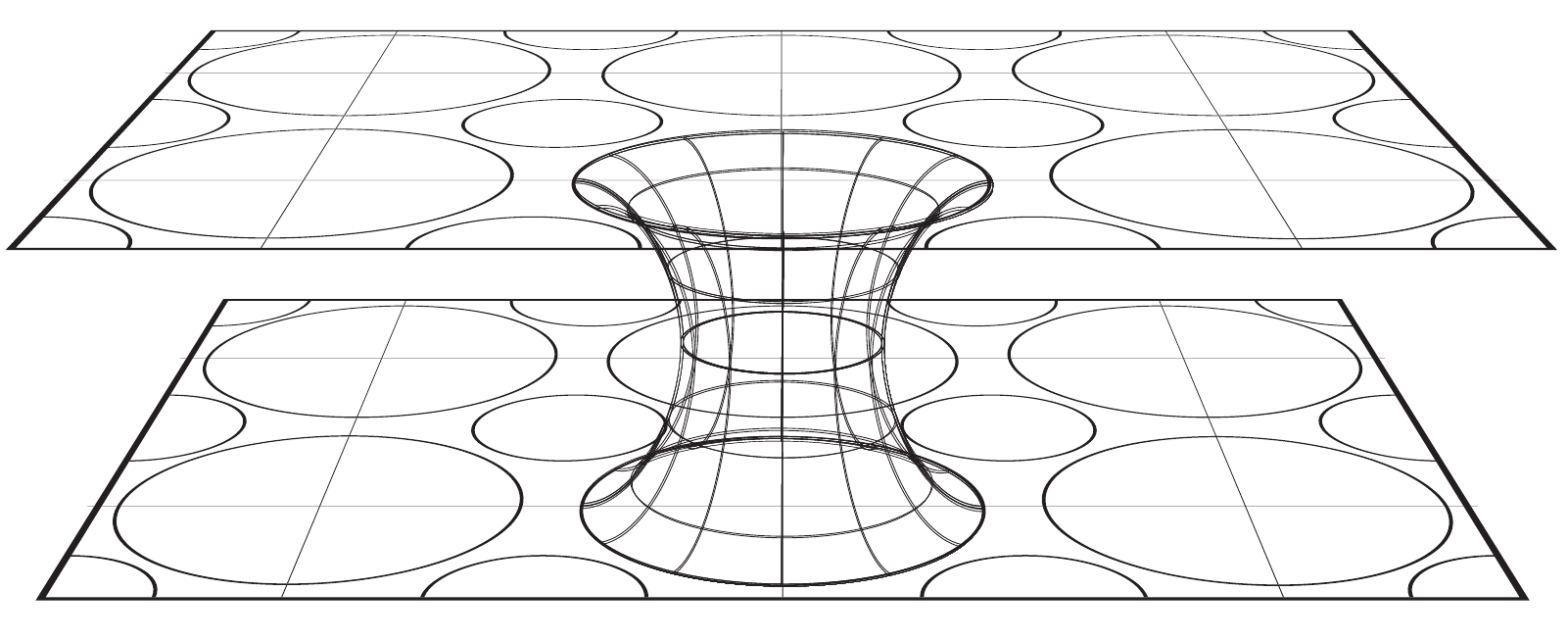}
\end{center}
\caption{A negatively curved tube connecting the $w=0$ and $w=1$ planes in the model space $\M$.} \label{fig:modelspace}
\end{figure}

\begin{theorem}\label{thm:modelanosov}
The geodesic flow on $(\M,g^\M_0)$ has a uniformly $\tau$-strictly invariant cone field and thus is Anosov.
\end{theorem}

\begin{proof}
We take as our cone field $C_{\xi_h,\xi_v}$ whose edges are generated by the standard basis vectors $\xi_h$ and $\xi_v$ on $\langle \dot{\varphi}^t \rangle^\perp$.
In order to show uniform $\tau$-strict invariance of these cones, we study the associated Jacobi fields, denoted by $j_h$ and $j_v$. With this notation, 
    \[
    \xi_h^t = D\varphi^t \xi_h = (j_h(t), j_h'(t)) 
    \]
and
    \[
    \xi_v^t = D\varphi^t \xi_v = (j_v(t), j_v'(t)), 
    \]
written in $\xi_h$-$\xi_v$ coordinates.

First, consider the evolution of the top edge of the cone $\xi_v^t = (j_v(t), j_v'(t))$, where $\xi_v^0 = (0,1)$ is in the  vertical subspace $V$. For $K\leq 0$ and non-negative values of $j$, the Jacobi equation shows $j'' = -Kj \geq 0$. This means $j_v'(t)\geq 1$ for $t \geq 0$ and $j_v(t)>0$ for $t>0$, so that $\xi_v^t$ lies strictly inside the cone field $C_{\xi_h,\xi_v}$ for $t>0$. Thus, the vertical vector is mapped strictly inside the first quadrant for $K \leq 0$. 

Next, consider the evolution of the bottom edge of the cone $\xi_h^t = (j_h(t), j_h'(t))$, where $\xi_h^0 = (1,0)$ is in the horizontal subspace $H$. For $K=0$, the Jacobi equation is $j'' = -Kj = 0$, so that $j'$ is constant and $j$ remains positive. For $K<0$ and positive values of $j$, the Jacobi equation is $j'' = -Kj > 0$, so that $j'$ is increasing. This means $j_h'(t)$ becomes positive and $j_h(t)$ remains positive, so that $\xi_h^t$ lies strictly inside the cone field $C_{\xi_h,\xi_v}$. Thus, the horizontal vector stays horizontal for $K=0$ and is mapped strictly inside the first quadrant once it encounters curvature $K<0$. 

Hence, the cone field $C_{\xi_h,\xi_v}$ is invariant since $K\leq0$, and it is strictly invariant as soon as every point in $S\M$ encounters some negative curvature under the flow. 

Let $x \in S\M$. Since $\D$ has the finite horizon property, there is a $\tau>0$ such that for all $x \in S\M$, the orbit $\varphi^tx$ spends some time in negative curvature between $\varphi^0 x = x$ and $\varphi^\tau x$. Hence, for every point $x \in S\M$, we have
    \[
    D\varphi^\tau (C_{\xi_h,\xi_v}(x)) \subset \Int\left(C_{\xi_h,\xi_v}(\varphi^\tau x)\right).
    \]
Although the space $S\M$ is non-compact, the metric $g^\M_0$ can be seen as the lift of a metric on a compact quotient $\M/(\Z^2 \times \Id)$. This compactness, together with the continuity of the cone field, guarantees the uniformity of the cone field's $\tau$-strict invariance.  Thus, by Theorem~\ref{thm:cone} $\varphi^t$ is Anosov.
\end{proof}

We note that our construction of $\M$ provides a way to produce Anosov geodesic flows on compact but non-embedded surfaces. Namely, take the model space $\M$ and mod out by the $\Z^2$ action of translations on $\R^2$ that preserve the set $\D$. For the set $\D$ pictured in Figure~\ref{fig:disks}, $\D/\Z^2$ gives a set of two closed disks on the torus $\R^2/\Z^2$.  Thus, for any subgroup $\Omega = a\Z \times b\Z$ with $a,b \in \mathbb{N}$, the set $\D/\Omega$ is a collection of $2ab$ disks on the torus $\R^2/\Omega$. Then $\M / (\Omega \times \Id)$ is a compact surface of genus $2ab+1$, and it has an Anosov geodesic flow induced by the metric $g^\M_0$. Note that the smallest possible genus for a surface produced in this way is $3$.

For compact manifolds, it is well known that the set of metrics with Anosov geodesic flows is $C^2$-open. In our setting of a non-compact manifold, $C^2$-openness of the Anosov property follows because the uniformly $\tau$-strictly invariant condition for continuous cone fields is an open property.

\begin{corollary}\label{cor:perturbations}
For all metrics $g^\M$ that are sufficiently $C^2$-close to $g_0^\M$ in the uniform norm, the geodesic flow on $(\M, g^\M)$ is Anosov.
\end{corollary}

\begin{proof}
Let $g$ be a metric on $\M$ that is $C^2$-close to $g^\M_0$ in the uniform norm. This implies that the maps $D\varphi^t_g$ and $D\varphi^t_{g^\M_0}$ are $C^0$-close, so that the cone field $C_{\xi_h,\xi_v}$ in the $g$ metric is also uniformly $\tau$-strictly invariant. So by Theorem~\ref{thm:cone}, the geodesic flow generated by $g$ is also Anosov.
\end{proof}

\section{A sequence of embedded surfaces}

We now will show that contained in the set $\mathcal{U}$ of Anosov metrics described above, there are metrics arising from embedded surfaces determined by the following mapping. 

Given two functions $R_1(s)$ and $R_2(s)$, define a one-parameter family of maps $X_s : \R^3 \to \R^3$ to be
    \[
    X_s(u, v, w) = (X_1(u, v, w, s), X_2(u, v, w, s), X_3(u, v, w, s))
    \]
with 
    \[
    X_1(u, v, w, s) =   \left( R_1(s) +( R_2(s)+w) \cos\left( \frac{v}{R_2(s)}\right) \right) \cos\left( \frac{  u }{R_1( s)}\right)
    \]
    \[
    X_2(u, v, w, s) = \left(R_1(s) +( R_2(s)+w) \cos\left( \frac{  v}{R_2( s)}\right) \right)\sin\left( \frac{  u }{R_1(s)}\right)
    \]
    \[
    X_3(u, v, w, s) =  ( R_2(s)+w) \sin\left( \frac{  v}{R_2(s)}\right) 
    \]
We are interested in the image of $\M$ under this map. Note that for a fixed value of $s$, the images of planes $w= \text{constant}$ are nested tori.

To understand this mapping, note that when $R_1 > R_2 > 0$, one has the standard torus embedding of the plane $w=0$ with the radii given by $R_1$ and $R_2$. The image of $\M$ under $X_s$ is the pair of concentric tori jointed by tubes (Figure~\ref{fig:embeddedsurface}). As $s \to \infty$, we will let the radii go to infinity so that the images of the planes $w=0$ and $w=1$ locally become arbitrarily close to flat. Under an additional condition on $R_1$ and $R_2$, the embedded surface produced by $X_s$ converges to a surface isometric to $\M$. 

\begin{figure}[ht]
\begin{center}
\includegraphics[width=3.5in]{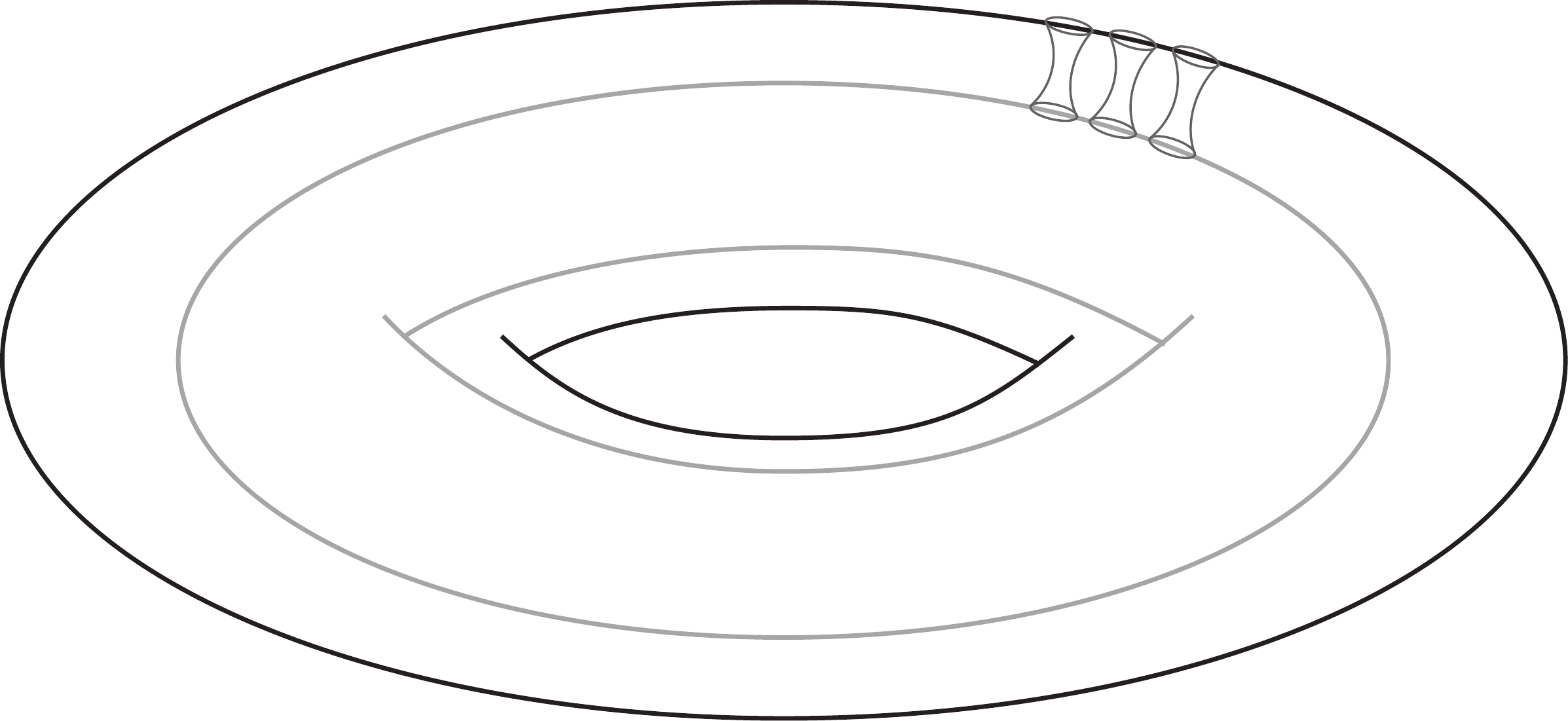}
\end{center}
\caption{An illustrative sketch of the embedded surface $X_s(\M)$ showing a few of the tubes connecting the inner and outer tori.}\label{fig:embeddedsurface}
\end{figure}

For any values of $R_1$ and $R_2$, the map $X_s$ is well-defined on $\R^3$, and the region $[0, 2 \pi R_1(s)] \times [0, 2 \pi R_2(s)] \times \{w_0\} \subset \R^3$ is mapped once around an embedded torus for any $w_0 \in \R$. In particular, the region $[0, 2 \pi R_1(s)] \times [0, 2 \pi R_2(s)] \times [0,1]$ will be mapped to a thickened torus in $\R^3$. However, this map may not be periodic when restricted to $\M$ due to the arrangement of the connecting tubes.

Let $g_s = X_s^* g_0$ be the pullback metric on $\R^3$ by the map $X_s$, and let $g^\M_s = g_s |_\M$ be the restriction of this metric to $\M$.

\begin{theorem} \label{thm:converge}
Assume that as $s\to \infty$, $R_1(s), R_2(s) \to \infty$  and $R_2(s)/R_1(s) \to 0$. Then  $\ds \lim_{s \to \infty} g^\M_s = g^\M_0$ in the $C^\infty$ topology.
\end{theorem}

\begin{proof}
This convergence result follows by examining the Jacobian matrix of partial derivatives $DX_s(u, v, w) $ given by 
    \[
    \left(
    \begin{array}{ccc}
     -\frac{\left(R_1(s)+ (R_2(s)+w) \cos \left(\frac{v}{R_2(s)}\right)\right) \sin \left(\frac{u}{R_1(s)}\right)} {R_1(s)}
     & -\frac{(R_2(s)+w) \sin \left(\frac{v}{R_2(s)}\right) \cos \left(\frac{u}{R_1(s)}\right)}{R_2(s)} 
     & \cos \left(\frac{v}{R_2(s)}\right) \cos \left(\frac{u}{R_1(s)}\right) \\
     \frac{\left(R_1(s)+ (R_2(s)+w) \cos \left(\frac{v}{R_2(s)}\right) \right) \cos \left(\frac{u}{R_1(s)}\right)}{R_1(s)} 
     & -\frac{(R_2(s)+w) \sin \left(\frac{v}{R_2(s)}\right) \sin \left(\frac{u}{R_1(s)}\right)}{R_2(s)} 
     & \cos \left(\frac{v}{R_2(s)}\right) \sin \left(\frac{u}{R_1(s)}\right) \\
     0 
     & \frac{(R_2(s)+w) \cos \left(\frac{v}{R_2(s)}\right)}{R_2(s)} 
     & \sin \left(\frac{v}{R_2(s)}\right) \\
    \end{array}
    \right)
    \]

Take two vectors $\zeta$ and $\eta$ in $T_p(\R^3)$  with Euclidean coordinates $(\zeta_1, \zeta_2, \zeta_3)$ and $(\eta_1, \eta_2, \eta_3)$, respectively. Their inner product in the Euclidan metric is
    \[
    \innerpr{\zeta, \eta}_{g_0} = \zeta_1 \eta_1+ \zeta_2 \eta_2 + \zeta_3 \eta_3. 
    \]
Their inner product in the $g_s$ metric, computed using the pullback definition $g_s = X_s^*g_0$, is
    \[
    \innerpr{\zeta, \eta}_{g_s} = \innerpr{DX \zeta, DX \eta}_{g_0} = 
    \left( 1+ \frac{(R_2+w) \cos \left(\frac{v}{R_2}\right)}{R_1} \right)^2 \zeta_1 \eta_1
    + \left( 1 +\frac{ w}{R_2} \right)^2 \zeta_2 \eta_2
    + \zeta_3 \eta_3
    \] 
The metric $g_s$ can be represented as a diagonal matrix $Q_s$ in standard Euclidean coordinates (i.e., so that $g_s(\zeta,\eta) = \zeta^T Q_s \eta$). The matrix $Q_s$ has components

\begin{align*}
Q_{11} &= 1 
        + 2w \cos \left(\frac{v}{R_2(s)}\right) \frac{1}{R_1(s)} \\
    & \qquad
        + 2\cos \left(\frac{v}{R_2(s)}\right) \frac{R_2(s)}{R_1(s)}
        + w^2 \cos ^2\left(\frac{v}{R_2(s)}\right) \frac{1}{R_1(s)^2}\\ 
    & \qquad 
        + 2w \cos^2\left(\frac{v}{R_2(s)}\right) \frac{R_2(s)}{R_1(s)^2}
        + \cos^2\left(\frac{v}{R_2(s)}\right) \frac{R_2(s)^2}{R_1(s)^2},\\
Q_{22} &= 1 + \frac{2w}{R_2(s)} + \frac{w^2}{R_2(s)^2}, \\
Q_{33} &= 1,
\end{align*}
with all other components zero. Under the assumptions that $R_1(s), R_2(s) \to \infty$ and $\frac{R_2(s)}{R_1(s)}\to 0$, we see that the $Q_s$ matrix converges to the identity matrix $Q_0$ as $s \to \infty$. Note that algebraically we need the condition $\frac{R_2(s)}{R_1(s)}\to 0$ because of two of the terms in $Q_{11}$.

Geometrically, one can see why we need this condition by looking at the length of the images of the  $v = \text{constant}$ lines on the torus. Their length varies from $2 \pi (R_1 - R_2) = 2 \pi R_1( 1 - \frac {R_2}{R_1})$ on the inner part of the torus to $2\pi (R_1+R_2) = 2\pi R_1( 1 + \frac {R_2}{R_1})$ on the outer part of the torus. For the metrics $g_s$ to converge to $g_0$ these lengths should all converge to $2 \pi R_1$. 

Higher derivatives of components of $Q_s$ (with respect to $u,v,w$) all converge to 0 as $s \to \infty$, as such derivatives either cause individual terms to vanish or else keep or increase the number of $R_j$  terms in the denominator while keeping the $R_j$ terms in the numerator as is. Hence, $g_s$ converges to $g_0$ in the $C^\infty$ topology. Since $g_s^\M$ and $g_0^\M$ are the restrictions of these metrics to the same surface in $\R^3$, we have $\lim_{s \to \infty} g^\M_s = g^\M_0$.
\end{proof}

Note that for any fixed $(u,v,w)$, $\ds \lim_{s \to \infty} DX_s(u,v,w) = \mat{0&0&1 \\ 1&0&0 \\ 0&1&0}$.
This is a rotation by $\pi/2$ in the $yz$-plane followed by a rotation by $\pi/2$ in the $xy$-plane. One can get some geometric intuition for this by noting that $X_s$ maps the origin to the point $(R_1+R_2, 0, 0)$, taking the $uvw$ frame to a frame obtained by these two rotations. In the limit, the $uvw$ frame at every point in $\R^3$ is related to its image frame in this way.

The pullback metric $g_s$ is always periodic on the slab $\R^2 \times [0,1]$ that contains $\M$. For $g_s$ to be periodic on $\M$, there needs to be a consistency between the fundamental region for the map $X_s$ and the fundamental region for the model space. 

\begin{theorem}\label{thm:periodic}
The metric $g_s$ is periodic on $\M$ with fundamental region $([0, m] \times [0, n] \times [0,1]) \cap \M$ for $m, n \in  \Z $ if and only if 
        $$
        R_1 (s) = \frac{m}{2 \pi}, \qquad R_2 (s) = \frac{n}{2 \pi}.
        $$
\end{theorem}

\begin{proof}
For the metric $g_s^\M = (X_s^* g_0)|_\M$ to be periodic, the map $X_s$ must be periodic on $\D \times [0,1]$, which has $\Z^2$ translational symmetry. The map $X_s$ is periodic in the $u$-direction with period $2 \pi R_1(s)$ and is periodic in the $v$-direction with period $2 \pi R_2(s)$.
Thus we have the conditions
$$
2 \pi R_1 = m, \qquad 2 \pi R_2 = n. 
$$
\end{proof}

\begin{proof}[Proof of Main Theorem]
By Corollary~\ref{cor:perturbations}, there is a $C^2$ open set $\mathcal{U}$ of Anosov metrics containing $g^\M_0$. Let $R_1(s) = \frac{s^2}{2\pi}$ and $R_2(s) = \frac{s}{2\pi}$. By Theorem~\ref{thm:converge}, $g^\M_s \rightarrow g^\M_0$ in the $C^\infty$ topology, so the geodesic flow on $(\M,g^\M_s)$ is Anosov for all sufficiently large $s$. For the sequence $s_n = n$, the metric $g^\M_{s_n}$ is periodic on $\M$ with fundamental region $([0,n^2] \times [0,n] \times [0,1]) \cap \M$ by Theorem~\ref{thm:periodic}. Thus, the metric $g^\M_{s_n}$ decends to a metric on $S_n = \M / (n^2\Z \times n\Z \times \Id)$, which is then isometric to a compact surface embedded in $\R^3$. This surface has $2n^3$ tubes and therefore genus $2n^3+1$.
\end{proof}

Note that other choices of the functions $R_1(s)$ and $R_2(s)$ will also work, but one must be careful to both satisfy the  conditions of Theorem~\ref{thm:converge} and have values of $s$ that satisfy Theorem~\ref{thm:periodic}. For instance, choosing $R_1(s) = \frac{1}{2\pi}s^{\sqrt{2}}$ and $R_2(s) = \frac{1}{2\pi}s$ satisfies the $R_1(s), R_2(s)$ conditions of Theorem~\ref{thm:converge}, but there are no values of $s$ that simultaneously make $2\pi R_1(s)$ and $2\pi R_2(s)$ integers.

In addition to proving that there exist  embedded surfaces in $\mathbb{R}^3$ for which the geodesic flows are Anosov, Donnay and Pugh show that such surfaces exist for all sufficiently large genus~\cite{Donnay-Pugh-finitehorizon}. Our construction gives only  a sequences of geni going to infinity for which there exist embedded surfaces with the desired properties. Modifications to our construction, which we will examine in a subsequent paper, should allow us to replicate the ``for all sufficiently large genus'' result.

\bibliographystyle{abbrv}
\bibliography{references.bib}

\end{document}